\documentclass[12pt]{article}
\usepackage{amsmath,amsfonts,amsthm,mathrsfs,indentfirst}
\usepackage{subeqnarray,comment}
\usepackage[]{hyperref}
\usepackage[left=3cm,right=3cm,top=2.5cm,bottom=2.5cm]{geometry}
\usepackage{graphics}
\usepackage{authblk}
\usepackage{xcolor}
\usepackage{colortbl}
\def\cor#1{{#1}}

\newtheorem{thm}{Theorem}

\newtheorem{lemma}{Lemma}

\def\nn{\nonumber}
\def\<{\left\langle}
\def\>{\right\rangle}
\def\nn{\nonumber}
\renewcommand\hat{\widehat}
\def\d{\mathrm{d}}
\def\bP{\mathbb{P}}
\def\hbP{{\hat{\bP}}}
\def\bnu{{\boldsymbol\nu}}
\def\p{\partial}
\def\hbP{{\hat{\bP}}}

\def\R{\mathbb{R}}
\def\Tau{\mathscr{T}}
\def\V{\mathscr{V}}
\def\sI{\mathscr{I}}
\def\hsI{\hat{\mathscr{I}}}

\def\M{\mathscr{M}}
\def\F{\mathscr{F}}
\def\bK{{\mathbf K}}
\def\hbK{{\hat{\bK}}}

\def\hphi{{\hat\phi}}
\def\hpsi{{\hat\psi}}
\def\hvarphi{\hat{\varphi}}
\def\hx{\hat{x}}
\def\x{x_1}
\def\y{x_2}
\def\z{x_3}
\def\grad{\nabla}
\def\hM{\hat{M}}
\def\O{\Omega}
\def\NC{\mathscr{NC}^h}
\def\Span{\mathrm{Span}}
\newcommand{\vertiii}[1]{{\left\vert\kern-0.25ex\left\vert\kern-0.25ex\left\vert
    #1\right\vert\kern-0.25ex\right\vert\kern-0.25ex\right\vert}}
\newcommand{\at}[1]{\kern-0.75ex\mid_{#1}}

\newcommand{\jump}[2]{{\left[#1\right]}_{{#2}}}
\newcommand{\djump}[2]{{\left[\kern-0.25ex\left[#1\right]\kern-0.25ex\right]}_{{#2}}}
\allowdisplaybreaks
\begin{document}
\title{Convergence analysis of a family of 14-node brick elements
\thanks{This project is supported by
NNSFC (Nos. 11301053,61033012,19201004,11271060,61272371), ``the Fundamental
Research Funds for the Central Universities''. Also partially supported by NRF of Korea (Nos. 2011-0000344, F01-2009-000-10122-0).}}
\author[a]{Zhaoliang Meng\thanks{Corresponding author: mzhl@dlut.edu.cn}}
\author[a,b]{Zhongxuan Luo}
\author[c,d]{Dongwoo Sheen}
\author[c]{Sihwan Kim}
\affil[a]{\it \small School of Mathematical Sciences,Dalian
University of Technology, Dalian, 116024, China}
\affil[b]{\it School of Software,Dalian
University of Technology, Dalian, 116620, China}
\affil[c]{\it
Department of Mathematics, Seoul National University,
Seoul 151-747, Korea.}
\affil[d]{\it
Interdisciplinary Program in
Computational Sciences \& Technology, Seoul National University,
Seoul 151-747, Korea.}

\maketitle
\begin{abstract}
    In this paper,  we will give \cor{a} convergence analysis for a family of 
    14-node elements which was proposed by I. M.
    Smith and D. J. Kidger[Int. J. Numer. Meth. Engng., 35:1263--1275, 1992]. The 14 DOFs are
    taken as the values at the eight vertices and the six
    face-centroids. For second-order elliptic problem\cor{s}, we will show that
    among all the Smith-Kidger 14-node elements,    
    Type 1, Type 2 and \cor{Type 5 elements provide
    optimal-order convergent solutions while Type 6 element gives one-order lower convergent solutions.
    Motivated by our proof, we also find that the order of convergence of 
    the Type 6 14-node nonconforming element improves to be optimal if 
    we change the DOFs into the values at the eight vertices and
    the integration values on the six faces. We also show that Type 1, Type 2 and
    Type 5 keep the optimal-order convergence if the integral DOFs on the six
    faces are adopted.}
\\[6pt]
\end{abstract}
\section{Introduction}
\cor{Among many three-dimensional brick elements, there have been well-known
simplest conforming elements such as the trilinear element, the 27-node element and
seredipity elements. For the} nonconforming case,
Rannacher-Turek \cite{rannacher-turek} presented the rotated trilinear
elements with the two types of 6 DOFs: 
the face-centroid values type and  the face integrals type.
Douglas-Santos-Sheen-Ye \cite{dssy-nc-ell} then
modified the element of Rannacher-Turek such that the face-centroid values
type and the face integrals type are identical, that is, the element fulfills
the mean value property ``the face-centroid value = the face average
integral''. Later Park-Sheen
presented a $P_1$-nonconforming finite element on cubic meshes which
has only 4 DOFs \cite{parksheen-p1quad}. Wilson
also defined a linear-order nonconforming brick element \cite{ciar,
wilson-brick} with 11 DOFs whose polynomial space consists of trilinear
polynomials plus $\{ 1-\x^2, 1-\y^2,  1-\z^2\}$ on $\hbK=[-1,1]^3$
(see \cite[Page 217, Remark 4.2.3]{ciar}).
All these three dimensional elements are of $O(h)$ convergence rate in energy
norm. 

In the direction of obtaining higher-order convergent nonconforming elements, 
Smith and Kidger \cite{smith-kidger-14node} successfully developed
three-dimensional brick elements of 14 DOFs by adding additional polynomials 
to $P_2$. They investigated six
most possible 14 DOFs elements systematically considering the Pascal
pyramid, and concluded that their Type 1 (as well as Type 2) and Type
6 elements are successful ones. The additional polynomial space
for Type 1 element is the span of the
four nonsymmetric cubic polynomials $\{\x\y\z, \x^2\y, \y^2\z, \z^2\x\}$
while that for Type 6 element is the span of 
$\{\x\y\z, \x\y^2\z^2,$ $\x^2\y\z^2,
\x^2\y^2\z \}.$ Only recently a new nonconforming brick
element of 14 DOFs with quadratic and cubic convergence in the
energy and $L_2$ norms, respectively, is introduced by 
Meng, Sheen, Luo, and Kim
\cite{2013-Meng-p-}, which has the same type of DOFs but has only
cubic polynomials added to $P_2$. And then, the authors compared
these 14-node elements numerically, see \cite{2013-Kim-p-}.
Numerical tests \cor{show} that at least for second-order elliptic
problems Meng-Sheen-Luo-Kim and some of Smith-Kidger elements are
convergent with optimal order or with lower order.

\cor{A convergence} analysis for Meng-Sheen-Luo-Kim element was \cor{reported} in
\cite{2013-Meng-p-} and is fairly easy because it satisfies the
patch test of Irons \cite{1977-Irons-p557}, which implies that a
successful $P_k$-nonconforming element needs to satisfy that on each
interface the jump of adjacent polynomials be orthogonal to
$P_{k-1}$ polynomials on the interface. Unfortunately, it was found
in mathematics that the patch test is neither necessary nor
sufficient, see \cite{2002-Shi-p221} and the \cor{references} therein. As
\cor{shown} in this paper, the Smith-Kidge elements can only pass \cor{a} lower
order patch test or can not pass it, but give optimal \cor{order} convergence
from our numerical results or lower convergence order. Thus, 
the convergence analysis for Smith-Kidger element \cor{seems to be} quite different
and complex. For the convergence of the nonconforming element which
\cor{fail to} pass the patch test, \cor{see the works of} Stummel, Shi, {\it etc.}
\cite{1979-Stummel-p449-471,1987-Shi-p391-405,1984-Shi-p1233-1246,2000-Shi-p97-102}.

In this paper, we will \cor{provide a} convergence analysis for Smith-Kidger \cor{elements}
for second-order elliptic \cor{problems}. We show that although the patch test
\cor{fails}, Type 1, 2 and 5 Smith-Kidger \cor{elements are of} optimal
convergence order, \cor{while} Type 6 element loses one order of accuracy.
Furthermore, we also present a new brick element with the same DOFs, which is also
convergent \cor{in optimal orders}. Finally, if the
value at the eight vertices and the integration \cor{values over} six faces are taken as the
DOFs, then we can show that Type 1, 2, 5, and 6 elements and the proposed new element
can get optimal convergence order, which implies that  Type 6 element
improves one order of accuracy.

The paper is organized as follows. In \cor{Section} 2, we will introduce
Smith-Kidger {elements} and give the basis functions firstly. In
\cor{Section} 3, we define \cor{an} interpolation operator and present
\cor{our} convergence analysis for Type 1 Smith-Kidger element. In \cor{Section} 4,
we will analyze the other elements and present the corresponding
error estimates very briefly. In \cor{Section} 5, a new 14-node brick
element is proposed. Finally, in \cor{Section} 6, we conclude our
results.

\section{The quadratic nonconforming brick elements}


Let $\hbK=[-1,1]^3$ and denote the vertices and {face-centroids} by
$V_j,1\leq j\leq 8,$ and $M_k,1\leq k\leq 6,$ respectively. (see Fig. \ref{fig:cube})

\begin{figure}[ht]
    \begin{center}
    \includegraphics{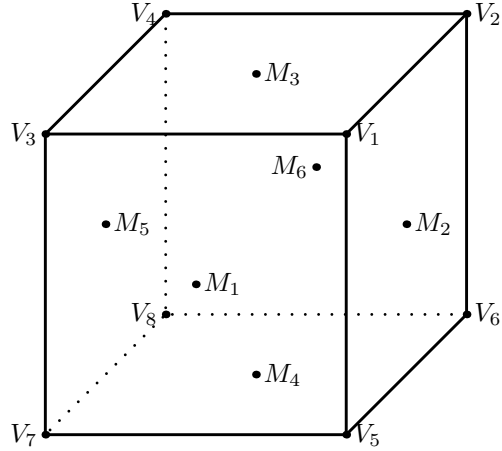}
    \end{center}
    \caption{$V_j$ denotes the vertices, $j=1,2,\ldots,8$, and $M_k$ denotes
      the \cor{face-centroids}, $k=1,2,\ldots,6.$}
    \label{fig:cube}
\end{figure}

Smith and Kidger \cite{smith-kidger-14node} defined the following six 14-node elements:
\begin{subeqnarray}
\hbP_{SK}^{(1)} &=& P_2(\hbK)\oplus\Span\{\hx_1\hx_2\hx_3, \hx_1^2\hx_2, \hx_2^2\hx_3, \hx_3^2\hx_1\}, \\
\hbP_{SK}^{(2)} &=& P_2(\hbK)\oplus\Span\{\hx_1\hx_2\hx_3, \hx_1\hx_2^2, \hx_2\hx_3^2, \hx_3\hx_1^2\}, \\
\hbP_{SK}^{(3)} &=& P_2(\hbK)\oplus\Span\{\hx_1\hx_2\hx_3, \hx_1^3, \hx_2^3, \hx_3^3\}, \\
\hbP_{SK}^{(4)} &=& P_2(\hbK)\oplus\Span\{\hx_1\hx_2\hx_3, \hx_1^2\hx_2\hx_3, \hx_1\hx_2^2\hx_3, \hx_1\hx_2\hx_3^2\}, \\
\hbP_{SK}^{(5)} &=& P_2(\hbK)\oplus\Span\{\hx_1\hx_2\hx_3, \hx_1^2\hx_2+\hx_1\hx_2^2, \hx_2^2\hx_3+\hx_2\hx_3^2, \hx_3^2\hx_1 + \hx_3\hx_1^2\},\\
\hbP_{SK}^{(6)} &=& P_2(\hbK)\oplus\Span\{\hx_1\hx_2\hx_3, \hx_1\hx_2^2\hx_3^2,
\hx_1^2\hx_2\hx_3^2, \hx_1^2\hx_2^2 \hx_3\},
\end{subeqnarray}
whose DOFs are the function values at the eight vertices and the six
face-centroids. They reported that Type 3 element fails and
is inadmissible. We also remark that Type 4 element is also inadmissible since
$(\hx_1^2-1)\hx_2\hx_3\in \hbP_{SK}^{(4)}$ vanishes at all these points. In \cite{2013-Kim-p-}, we observe that Type 1 (and 2)
and Type 5 elements give optimal convergence results both in $L^2$-
and $H^1$-norms at least for second-order elliptic problems,
while Type 6 element loses one order of accuracy in each norm.

In what follows, we will give \cor{an} error estimate for \cor{Type 1}
Smith-Kidger element \cor{in detail};  error estimates for the other
types can be obtained completely \cor{analogously} and thus are stated very briefly.

\cor{To begin with, 
denote by $V_m, m = 1,\cdots, 8,$ the eight vertices 
$(j,k,l), |j| = |k| = |l|= 1, j,k,l\in \mathbb Z,$
and and by $V_m, m = 1,\cdots, 6,$ the six face-centroids
$(j,k,l), |j| + |k| + |l|= 1, j,k,l\in \mathbb Z.$
The basis functions corresponding to the eight vertices $V_j, j = 1,\cdots,8,$ 
are denoted by $\hphi^V_{V_j},$ and those corresponding to 
the six face-centroids  by $\hphi^F_{M_j}, j = 1, \cdots, 6.$} 

\cor{To describe the brick elements in a uniform fashion, set 
\begin{eqnarray}\label{eq:r}
\begin{split}
r_0(\hx_1,\hx_2,\hx_3) = \hx_1\hx_2\hx_3,~ r_1(\hx_1,\hx_2,\hx_3) =\hx_1\hx_3^2
,\\
r_2(\hx_1,\hx_2,\hx_3) =\hx_2 \hx_1^2,~r_3(\hx_1,\hx_2,\hx_3) = \hx_3 \hx_2^2,
\end{split}
\end{eqnarray}
so that
\[
\hbP_{SK}^{(1)} = P_2(\hbK)\oplus\Span\{
r_0(\hx_1,\hx_2,\hx_3),~ r_1(\hx_1,\hx_2,\hx_3),~r_2(\hx_1,\hx_2,\hx_3),~r_3(\hx_1,\hx_2,\hx_3)\}. 
\]
Then the basis functions for Type 1 Smith-Kidger elements are
given as follows:}

\noindent for $(j,k,l)  = V_m, m = 1,\cdots, 8,$
\begin{eqnarray}\label{eq:phiV}
\begin{split}
&&\hphi^V_{(j,k,l)} (\hx_1,\hx_2,\hx_3) =\frac1{16}\left[-1+\hx_1^2+\hx_2^2+\hx_3^2)\right]+\frac{1}{8}\left[jk\hx_1\hx_2+jl\hx_1\hx_3+kl\hx_2\hx_3\right.
  \\
&&\qquad\left. +jklr_0(\hx_1,\hx_2,\hx_3) +
  r(\hx_1,\hx_2,\hx_3) \right],
\end{split}
\end{eqnarray}
and, for $(j,k,l) =M_m,m=1,\cdots,6,$
\begin{eqnarray}\label{eq:phiF}
\hphi^F_{(j,k,l)}(\hx_1,\hx_2,\hx_3
  )=\frac{1}{4}+\frac{1}{2}\ell(\hx_1,\hx_2,\hx_3) +\frac{1}{4}q(\hx_1,\hx_2,\hx_3)
  -\frac{1}{2} r(\hx_1,\hx_2,\hx_3),
\end{eqnarray}
where the linear, quadratic, and the remaining higher-order terms are defined by
\begin{subeqnarray}\label{eq:ell-q-r}
&&\ell(\hx_1,\hx_2,\hx_3)= j\hx_1 + k\hx_2 + l\hx_3,\\
 && q(\hx_1,\hx_2,\hx_3)
= -(\hx_1^2+\hx_2^2+\hx_3^2)+ 2 (j\hx_1^2 + k\hx_2^2+ l\hx_3^2),
\\
&&        
 r(\hx_1,\hx_2,\hx_3)= jr_1(\hx_1,\hx_2,\hx_3) + kr_2(\hx_1,\hx_2,\hx_3) + lr_3(\hx_1,\hx_2,\hx_3).
\end{subeqnarray}

Assume that $\O\in \R^3$ is a parallelepiped domain with boundary
$\Gamma$. Let $(\Tau_h)_{h>0}$ be a regular family of \cor{triangulations} of $\O$ into
parallelepipeds $\bK_j,j=1,2,\ldots,N_\bK$, where $h=\max_{\bK\in \Tau_h}h_\bK$ with
$h_\bK=\text{diam}(\bK)$. For each $\bK\in\Tau_h$, let $T_\bK: \hbK\rightarrow
\mathbb R^3$ be an invertible affine mapping such that
\[
\bK = T_\bK(\hbK),
\]
and set $\phi_{\bK} = \hphi\circ T_\bK^{-1}: \bK \rightarrow \mathbb R$ for
all $\hphi \in \hbP_{SK}^{(1)}$, whose collection will be designated by
\[
\bP_\bK = \Span\{\phi_{\bK}\mid\,\hphi\in\hbP_{SK}^{(1)} \}.
\]
Let $N_V$ and $N_F$ denote the numbers of vertices and faces, respectively. Then set
\begin{eqnarray*}
    && \V_h=\{V_1,V_2,\cdots,V_{N_V}:\quad \text{the set of all vertices of $\bK\in\Tau_h$}\},\\
    && \F_h=\{F_1,F_2,\cdots,F_{N_F}:\quad \text{the set of all faces of
    $\bK\in\Tau_h$}\},\\
    && \F_h^i=\{F_1,F_2,\cdots,F_{N_F^i}:\quad \text{the set of all interior faces of $\bK\in\Tau_h$}\},\\
    && \M_h=\{M_1,M_2,\cdots,M_{N_F}:\quad \text{the set of all face-centroids
    on $\F_h$}\},\\
    && \F_h^{(1)}=\{F\in\F_h:\quad \text{the set of all faces with outward
    normal $(\pm 1,0,0)$}\},\\
    && \F_h^{(2)}=\{F\in\F_h:\quad \text{the set of all faces with outward
    normal $(0,\pm 1,0)$}\},\\
    && \F_h^{(3)}=\{F\in\F_h:\quad \text{the set of all faces with outward
    normal $(0,0,\pm 1)$}\}.
\end{eqnarray*}
Obviously we have $\F_h=\F_h^{(1)}\cup\F_h^{(2)}\cup\F_h^{(3)}$.

We consider the following nonconforming finite element spaces:
\begin{eqnarray*}
    && \NC=\{\phi: \O\rightarrow \R |~ \phi|_\bK\in \bP_\bK \forall
  \bK\in\Tau_h, \phi \text{ is continuous
    at all } V_j\in \V_h, M_k\in \M_h \},\\
    && \NC_0=\{\phi\in\NC |~ \phi(V)=0\,\,\forall V_j\in \V_h\cap \Gamma\text{ and }\phi(M)=0\,\, \forall
 M_k\in \M_h\cap\Gamma\}.
\end{eqnarray*}

\section{The interpolation operator and convergence analysis}
In this section we will define an interpolation operator and analyze convergence in the case of
Dirichlet boundary value problems. The case of Neumann boundary value problem is quite similar and the
results will be omitted.

\subsection{The second order elliptic problem}

Denote by $(\cdot,\cdot)$ the $L^2(\O)$ inner product and \cor{by
  $\<\cdot,\cdot\>$} 
the duality pairing between $H^{-1}(\O)$ and $H^1_0(\O)$, which
is an extension of the duality paring between $L^2(\O)$ \cor{and itself}. By $\|\cdot\|_k$ and
$|\cdot|_k$ we adopt the standard notations for the norm and seminorm for the Sobolev spaces $H^k(\O)$.
Consider then the following Dirichlet boundary value problem:
\begin{subequations}\label{eq:elliptic}
\begin{eqnarray}
    -\Delta u &=&f, \quad \O, \\
    u&=&0,\quad \Gamma,
\end{eqnarray}
\end{subequations}
with  $f\in H^1(\O)$.
We will further assume that the coefficients
are sufficiently smooth and that the elliptic problem
\eqref{eq:elliptic} has an $H^3(\O)$-regular solution such that
$\|u\|_{3} \le C \|f\|_1$.
 The weak problem is then given as usual:
find $u\in H_0^1(\O)$ such that
\begin{equation}
a(u, v) = \<f,v\>, \quad v\in H^1_0(\O),
    \label{eq:weak}
\end{equation}
where $a: H^1_0(\O)\times H^1_0(\O) \rightarrow \R$  is the bilinear
form defined by $a(u,v) = (\nabla u,\nabla v)$ for all $u,v\in H^1_0(\O)$. The nonconforming \cor{Galerkin} method for Problem
\eqref{eq:elliptic} states as follows: find $u_h\in\NC_0$ such that
\begin{equation}
    a_h(u_h, v_h) = \<f, v_h\>,\quad v_h\in\NC_0,
    \label{eq:solution}
\end{equation}
where
$$
a_h(u, v) =\sum_{\bK\in \Tau_h}a_\bK(u,v),
$$
with $a_\bK$ being the restriction of $a$ to $\bK$.

Notice that in order to have point values defined properly we need to recall
the following Sobolev embedding theorem
\begin{eqnarray*}
	W^{m,p}(\O)\longrightarrow C^0(\bar{\O}),\ \text{if}\
    \frac{1}{p}-\frac{m}{d}<0.
\end{eqnarray*}
Thus we should have $p>\frac{3}{m}$. For a given cube $\bK\in \Tau_h$, define the local interpolation operator
$\Pi_\bK : W^{1,p}(\bK)\cap H^1_0(\O)\longrightarrow \hbP_{SK}^{(1)},~ p>3,$ by
$$
\Pi_\bK\phi(V_i)=\phi(V_i),\quad \Pi_\bK\phi(M_j)=\phi(M_j)
$$
for all vertices $V_i$ and face-centroids $M_j$ of $\bK$.
The global interpolation operator $\Pi_h$: $W^{1,p}(\O)\cap H^1_0(\O)\rightarrow\NC_0$
is then defined through the local interpolation operator $\Pi_\bK$ by $\Pi_h|_\bK=\Pi_\bK$ for all $\bK\in\Tau_h$. Since $\Pi_h$
preserves $P_2$ for all $\bK\in\Tau_h$, it follows from the Bramble-Hilbert Lemma that
\begin{eqnarray}\label{eq:hilbert}
    \begin{split}
\sum_{\bK\in\Tau_h}\|\phi-\Pi_h\phi\|_{0,\bK}+h\sum_{\bK\in\Tau_h}\|\phi-\Pi_h\phi\|_{1,\bK}\le
Ch^k\|\phi\|_k,
\\
\phi\in W^{k,p}(\O)\cap H_0^1(\O),1\le k\le 3.
\end{split}
\end{eqnarray}

We now consider the energy-norm error estimate and first consider the
following Strang lemma \cite{strang-fix}.
\begin{lemma}
    Let $u\in H^1(\O)$ and $u_h\in \NC_0$ be the solutions of Eq.
    \eqref{eq:weak} and Eq. \eqref{eq:solution}, respectively.
Then
\begin{eqnarray}\label{eq:strang}
    \|u-u_h\|_h\leq C\Big\{\inf_{v_h\in \NC_0}\|u-v_h\|_h+\sup_{w_h\in\NC_0}
    \frac{|a_h(u,w_h)-\< f,w_h\>|}{\|w_h\|_h}\Big\}.
\end{eqnarray}
    \label{lem:strang}
\end{lemma}
Here, and in what follows, $\|\cdot\|_h$ denotes the usual broken energy norm
such that
$\|v\|_h = \sqrt{ a_h(v,v)}.$

Due to \eqref{eq:hilbert}, the first term in the right side
of \eqref{eq:strang} is bounded by
\begin{eqnarray}\label{eq:int_error}
    \inf_{v_h\in \NC_0}\|u-v_h\|_h\leq \|u-\Pi_h u\|_h\leq
    C h^s|u|_{s+1},~1<s\leq 2.
\end{eqnarray}

Denote by $f_{jk}$ the trace of $f\at{\bK_j}$ on $F_{jk}=\p\bK_j\cup\p\bK_k$
if it is nonempty. Similarly, 
the face $F_{jk}$ will designate the boundary of $K_j$ common with that of $K_k.$

Now let us bound the second term of the right side of \eqref{eq:strang} which
denotes the consistency error.
For a given cube $\bK\in \Tau_h$, denote by $F_K^{x_1+}$ and
$F_K^{x_1-}$ the face
of $\bK$ with outward normal $(1,0,0)$ and $(-1,0,0)$, respectively. Similarly,
we denote the other faces by
$F_K^{x_2+},F_K^{x_2-},F_K^{x_3+},$ and $F_K^{x_3-}$ so that
$\partial
\bK=\{F_K^{x_1+},F_K^{x_1-},F_K^{x_2+},F_K^{x_2-},F_K^{x_3+},F_K^{x_3-}\}$. Thus,
integrating by parts elementwise, we have
\begin{eqnarray}
    a_h(u,w_h)-\< f,w_h\>&=&\sum_{\bK\in\Tau_h}\<\frac{\p u}{\p\bnu},w_h\>_{\partial \bK} \nn \\
    &=&
    \sum_{\bK\in\Tau_h}\int_{F_K^{x_1+}\cup F_K^{x_1-}}\frac{\p u}{\p\bnu} w_h~\d
    s+\sum_{\bK\in\Tau_h}\int_{F_K^{x_2+}\cup F_K^{x_2-}}\frac{\p u}{\p\bnu}
    w_h~\d s  \nn \\
   &&+\sum_{\bK\in\Tau_h}\int_{F_K^{x_3+}\cup F_K^{x_3-}}\frac{\p u}{\p\bnu}
   w_h~\d s=: E_1+E_2+E_3, \label{eq:ExEyEz}
\end{eqnarray}
where $\bnu$ is the unit outward normal to $\bK$.

Before proceeding, we need the following lemma.
\begin{lemma}\label{lem:orth}
    By $F_k$ denote the face containing the centroid $M_k$ and by
    $V_j^{F_k},j=1,2,3,4,$ denote the vertices on the \cor{face} $F_k$. If
    $p\in\hbP_{SK}^{(1)},$  $p(V_j^{F_k})=0,j=1,2,3,4,$ and $p(M_k)=0$, then
    \begin{eqnarray}\label{eq:orth_property1}
    \int_{F_k}p(\x,\y,\z)\d s=0,\quad
    k=1,2,\ldots,6.
    \end{eqnarray}
\end{lemma}
\begin{proof}
 Without loss of generality, we assume that $M_1=(1,0,0)$. In this case,
 we have $p\in\hbP_{SK}^{(1)}$ and $p(1,\pm 1,\pm 1)= p(1,0,0)=0.$ We need to prove that
\begin{eqnarray}\label{eq:orth_property}
    \int_{F_1}p(1,\y,\z)\,\d \y\d\z=0.
\end{eqnarray}
It follows from $p(1,\pm 1,\pm 1)=0$ that
\begin{eqnarray}
    p(1,\y,\z)=l_1(\y,\z)(\y^2-1)+l_2(\y,\z)(\z^2-1), \quad l_j\in P_1(\mathbb R^2),~j=1,2.
    \label{eq:orth}
\end{eqnarray}
Set
$$
l_j(\y,\z)=a_j\y+b_j\z+c_j,\quad j=1,2.
$$
Then $p(1,0,0)=0$ implies that $c_1+c_2=0,$ which reduces \eqref{eq:orth} to
\begin{eqnarray*}
    p(1,\y,\z)=(a_1\y+b_1\z)(\y^2-1)+(a_2\y+b_2\z)(\z^2-1)+c_1(\y^2-\z^2).
\end{eqnarray*}
Since
\begin{equation}
    \hbP_{SK}^{(1)}|_{\x=1}=\Span\{1,\y,\z,\y^2,\y\z,\z^2,\y^2\z\},
    \label{eq:spanx}
\end{equation}
invoking $p\in \hbP_{SK}^{(1)}$, we have
\begin{eqnarray*}
    a_1=a_2=b_2=0,
\end{eqnarray*}
which leads to
\begin{eqnarray}\label{eq:poly_form}
    p(1,\y,\z)=b_1\z(\y^2-1)+c_1(\y^2-\z^2).
\end{eqnarray}
It follows from \eqref{eq:poly_form} that
\eqref{eq:orth_property1} holds. This completes the proof.
\end{proof}
This lemma implies that Type 1 element can pass lower order patch
test (test functions are in $P_0$ not $P_1$), which will lead to a
convergence solution for the second order elliptic problems, but the
convergence order is not optimal. To bound $E_1,E_2,E_3$, we also
need some interpolation operators.

\subsection{Some interpolation and projection operators}
For the reference element $\hbK=[-1,1]\times [-1,1]\times[-1,1]$, consider
the interpolation problem on the face of $F_{\hbK}^{x_1+}$: the interpolation
points are $(1,1,1)$, $(1,1,-1)$, $(1,-1,1)$, $(1,-1,-1)$, and $(1,0,0)$, which are the four
vertices and the centroid of $F_{\hbK}^{x_1+},$ with the interpolation space
${\hat Q}_1^*( F_{\hbK}^{x_1+} ),$ where 
\begin{equation*}
{\hat Q}_1^*( F_{\hbK}^{x_1+} ):=\Span\{1,\hx_2,\hx_3,\hx_2\hx_3,\hx_3^2\} \subset
\hbP_{SK}^{(1)}|_{\hx_1=1}
\end{equation*}
is an enriched bilinear polynomial space on the face
$F_{\hbK}^{x_1+}$ (see \eqref{eq:spanx}).

The above interpolation problem has a solution by using the bubble function
\begin{eqnarray*}
b(\hx_2,\hx_3)=1-\hx_3^2,
\end{eqnarray*}
and the standard bilinear interpolation basis functions
\begin{eqnarray*}
 &&   q_1(\hx_2,\hx_3) =\frac14 (1+\hx_2)( 1 + \hx_3),\quad
    q_2(\hx_2,\hx_3)=\frac14(1 - \hx_2)( 1+ \hx_3) ,\\
&& q_3(\hx_2,\hx_3)= \frac14 (1 - \hx_2)(1 - \hx_3),\quad
    q_4(\hx_2,\hx_3)= \frac14 (1+ \hx_2)( 1- \hx_3),\\
\end{eqnarray*}
as follows:
\begin{eqnarray*}
\hvarphi_j =  q_j - \frac14b,\, j =1,\cdots,4; \quad 
\hvarphi_5=b.
\end{eqnarray*}
Thus for a \cor{continuous} function $f$ defined on the face
$F_{\hbK}^{x_1+}$, the interpolation
polynomial is given by
\begin{eqnarray}\label{eq:hsIF}
\hsI_F^{x_1+}f=f(1,1,1)\hvarphi_1+f(1,-1,1)\hvarphi_2+f(1,-1,-1)\hvarphi_3+f(1,1,-1)\hvarphi_4+f(1,0,0)\hvarphi_5.
\end{eqnarray}
And then we can also define the interpolation operator on
the opposite face with the same space and denote it by
$\hsI_F^{x_1-}$. Similarly, 
define the interpolation operators on the other
faces of $\hbK$ with the corresponding spaces:
\begin{eqnarray*}
&&{\hat Q}_1^*( F_{\hbK}^{x_2} ):=\Span\{1,\hx_3,\hx_1,\hx_3\hx_1,\hx_1^2\} \subset
\hbP_{SK}^{(1)}|_{\hx_2=\pm 1},\\
&&{\hat Q}_1^*( F_{\hbK}^{x_3} ):=\Span\{1,\hx_1,\hx_2,\hx_1\hx_2,\hx_2^2\} \subset
\hbP_{SK}^{(1)}|_{\hx_3=\pm 1}
\end{eqnarray*}
 and denote them by
$\hsI_F^{x_2+},\hsI_F^{x_2-},\hsI_F^{x_3+},\hsI_F^{x_3-}$, respectively.

For a given $\bK\in\Tau_h$, we can define the interpolation operator
$\sI_F^{x_i\pm}$ by $\sI_F^{x_i\pm}=\hsI_F^{x_i\pm}\circ
T_\bK^{-1}$.
Notice that $\jump{\sI_F^{x_i+}w_h}{F}= 0$ for all interior faces $F$ for every $w_h\in \NC_0.$
Moreover, the above interpolation operators
preserve linear polynomials on each surface as stated in the following lemma.
\begin{lemma}\label{lem:sIF}
$\sI_F^{x_i\pm}$ 
map $w_h\in \NC_0$ such that their images across interior faces are continuous
for all interior faces $F$. Moreover, they
preserve bilinear polynomials on faces. 
\end{lemma}

Moreover, the above interpolation has the following interesting property:
\begin{lemma}\label{lem:sIF-wh}
For all $w_h\in \NC_0$ and $\bK\in\Tau_h,$ 
\begin{eqnarray}\label{eq:Fx}
w_h|_{F_K^{x_i+}}-\sI_F^{x_i+}(w_h\at{\bK})=w_h|_{F_K^{x_i-}}-\sI_F^{x_i-}(w_h\at{\bK}),
\quad i=1,2,3.
\end{eqnarray}
\end{lemma}
\begin{proof}
	We only prove the case of $i=1$ in Eq. \eqref{eq:Fx}
	which suffices to prove the statement on the reference element $\hbK$:
\begin{eqnarray}    \label{eq:Fhx}
    \hat{w}_h|_{\hat{F}_{\hbK}^{x_1+}}-\hsI_{\hat{F}}^{x_1+}(\hat{w}_h)=\hat{w}_h|_{\hat{F}_{\hbK}^{x_1-}}-\hsI_{\hat{F}}^{x_1-}(\hat{w}_h)\quad\forall
    \hat{w}_h\in \hbP_{SK}^{(1)}.
\end{eqnarray}
Due to the interpolation property, 
$\hat{w}_h|_{\hat{F}_{\hbK}^{x_1+}}-\hsI_{\hat{F}}^{x_1+}(\hat{w}_h)=0$
 for all  $\hat{w}_h\in
{\hat Q}_1^*( F_{\hbK}^{x_1+}),$ and the same is true if
$x_1^+$ is replaced by $x_1^-.$
Thus, it suffices to show that \eqref{eq:Fx} holds for all  $\hat{w}_h\in
\hbP_{SK}^{(1)}|_{\hx_1=1} \setminus {\hat Q}_1^*(
F_{\hbK}^{x_1+}),$ which
is nothing but $\Span\{\hx_2^2, \hx_2^2 \hx_3\}.$ Since both 
$\hx_2^2$ and $\hx_2^2 \hx_3$ are independent of $\hx_1$, it is obvious that
\eqref{eq:Fhx} holds for each of them.
This proves the lemma. 
\end{proof}

Define
$RQ=\Span\{1,\hx_1,\hx_2,\hx_3,\hx_1^2-\hx_2^2,\hx_1^2-\hx_3^2\}$.
For the reference element $\hbK$, let $R_{\hbK}: H^2(\hbK)\rightarrow RQ$ be an
interpolation operator defined by
\begin{eqnarray*}
    R_{\hbK}\hphi (\hat{M}_j)=\hphi (\hat{M}_j), j=1,\ldots,6
\end{eqnarray*}
for all $\hphi\in H^2(\hbK)$. It is exactly the so-called rotation element.
Obviously,
\begin{eqnarray*}
    R_{\hbK}\hphi =\sum_{i=1}^6\hphi (\hM_i)\hpsi_i (\hx,\y,\z)
\end{eqnarray*}
where
\begin{eqnarray*}
	\begin{cases}
	\hpsi_i=\frac{1}{6}\left(1+3\hx_i+\sum_{1\leq j\leq
	3,j\neq i} (\hx_i^2-\hx_j^2) \right),\\
	\hpsi_{7-i}=\frac{1}{6}\left(1-3\hx_i+\sum_{1\leq j\leq
	3,j\neq i} (\hx_i^2-\hx_j^2) \right),
\end{cases} i=1,2,3.
\end{eqnarray*}
It is easy to notice that for $i=1,2,\ldots,6$
\begin{eqnarray*}
	&&\hpsi_i\at{\hx_j=1}=\hpsi_i\at{\hx_j=-1}, \quad
	\text{if $j\neq i$ and $j\neq 7-i$}
\end{eqnarray*}
and for $i=1,2,3$
\begin{eqnarray*}
	&&\hpsi_i\at{\hx_i=1}=\hpsi_{7-i}\at{\hx_i=-1}=1-\frac{1}{6}\sum_{1\leq j\leq
	3,j\neq i}\hx_j^2,  \\
	&&\hpsi_i\at{\hx_i=-1}=\hpsi_{7-i}\at{\hx_i=1}=-\frac{1}{6}\sum_{1\leq j\leq
	3,j\neq i}\hx_j^2. 
\end{eqnarray*}

Thus we have
\begin{eqnarray*}
    R_{\hbK}\hphi|_{\hx_j=1} &=&\sum_{i=1}^6\hphi
    (\hM_i)\hpsi_i\at{x_j=1}\\
    &=&\sum_{1\leq i\leq 6,i\neq j,i\neq 7-j}\hphi
    (\hM_i)\hpsi_i\at{x_j=1}+\hphi(\hM_j)
    \left(1-\frac{1}{6}\sum_{1\leq i\leq
	3,i\neq j}\hx_i^2\right)\\
    &&-\hphi(\hM_{7-j})\left(\frac{1}{6}\sum_{1\leq i\leq
	3,i\neq j}\hx_i^2\right)\\
    &:=&
    \Theta(\hbK,\hphi,\{\hx_1,\hx_2,\hx_3\}\setminus\hx_j)+\hphi(\hM_j),
\end{eqnarray*}
and
\begin{eqnarray*}
    R_{\hbK}\hphi|_{\hx_j=-1} &=&\sum_{i=1}^6\hphi
    (\hM_i)\hpsi_i\at{x_j=-1}   \\
    &=&\sum_{1\leq i\leq 6,i\neq j,i\neq 7-j}\hphi
    (\hM_i)\hpsi_i\at{x_j=1}-\hphi(\hM_j)
    \left(\frac{1}{6}\sum_{1\leq i\leq
	3,i\neq j}\hx_i^2\right)\\
    &&-\hphi(\hM_{7-j})\left(1-\frac{1}{6}\sum_{1\leq i\leq
	3,i\neq j}\hx_i^2\right)\\
    &=&
    \Theta(\hbK,\hphi,\{\hx_1,\hx_2,\hx_3\}\setminus\hx_j)+\hphi(\hM_{7-j}),
\end{eqnarray*}
For any given $\bK\in\Tau_h$, we can define the interpolation operator
$R_{\bK}:=R_{\hbK}\cdot T_K^{-1}$. Denote by $M_K^{x_j+}$
and $M_K^{x_j-}$ the
centroids of the faces $F_K^{x_j+}$ and $F_K^{x_j-}$,
respectively. Then for any $\phi\in H^2(\bK)$, we have
\begin{eqnarray*}
	&&R_{\bK}\phi|_{F_K^{x_j+}}=\Theta(\bK,\phi,\{x_1,x_2,x_3\}\setminus
	x_j)+\phi(M_K^{x_j+}),\\
    &&R_{\bK}\phi|_{F_K^{x_j-}}=\Theta(\bK,\phi,\{x_1,x_2,x_3\}\setminus
	x_j)+\phi(M_K^{x_j-}).
\end{eqnarray*}


\subsection{The error estimates}
Turn to bound $|E_1| + |E_2| + |E_3|$ in \eqref{eq:ExEyEz}.
Below, we will give an estimate of $|E_1|$ in detail while 
similar estimates of $|E_2|$ and $|E_3|$ will be omitted.

It is easy to notice that for any $F\in \partial
\bK'\cap\partial\bK''\cap \F_h^{(1)}\neq \emptyset $ and $w\in\NC_0$, we have
\begin{eqnarray*}
	&&\int_F \nabla u \big(w|_{\bK'}-w|_{\bK''}\big)\d s\\
	&=& \int_F \nabla u
	\big((w|_{\bK'}-\sI_F^{x_1}(w))-(w|_{\bK''}-\sI_F^{x_1}(w))\big)\d s\\
	&=& \int_F \big(\nabla u-M_F(\nabla u)\big)
	\big((w|_{\bK'}-\sI_F^{x_1}(w))-(w|_{\bK''}-\sI_F^{x_1}(w))\big)\d s
\end{eqnarray*}
where $M_F(\nabla u)$ denotes the value of
$\nabla u$ at the centroid of $F$. The second equality holds
due to the orthogonality \eqref{eq:orth_property1}.
Hence we have
\begin{eqnarray*}
	E_1 &=&\sum_{\bK\in\Tau_h}\sum_{i=1}^3\int_{F_K^{x_1+}\cup
	F_K^{x_1-}}\frac{\partial u}{\partial x_i}w\nu_i\d s\\ 
      &=&
      \sum_{\bK\in\Tau_h}\sum_{i=1}^3\left(\int_{F_K^{x_1+}}\frac{\partial u}{\partial x_i}
    (w-\sI_F^{x_1+}(w))\nu_i\d s +\int_{F_K^{x_1-}}\frac{\partial u}{\partial x_i}
    (w-\sI_F^{x_1-}(w))\nu_i\d s\right)\\
    &=&
    \sum_{\bK\in\Tau_h}\sum_{i=1}^3\left(\int_{F_K^{x_1+}}\left(\frac{\partial
    u}{\partial x_i}-M_{F_K^{x_1+}}(\frac{\partial u}{\partial x_i})\right)
    (w-\sI_F^{x_1+}(w))\nu_i\d s\right.\\
    &&\left.+\int_{F_K^{x_1-}}\left(\frac{\partial
    u}{\partial x_i}-M_{F_K^{x_1-}}(\frac{\partial u}{\partial x_i})\right)
    (w-\sI_F^{x_1-}(w))\nu_i\d s\right),
\end{eqnarray*}
where $\bnu=(\nu_1,\nu_2,\nu_3)^T$ is the outward normal derivative of
$F$.
Thus due to \eqref{eq:Fx}, we arrive at
\begin{eqnarray*}
	E_1 &=&
	\sum_{\bK\in\Tau_h}\sum_{i=1}^3\left(\int_{F_K^{x_1+}}\left(\frac{\partial u}{\partial x_i}
    -\Theta(\bK,\frac{\partial u}{\partial
    x_i},x_2,x_3)-M_{F_K^{x_1+}}(\frac{\partial u}{\partial x_i})\right)
    (w-\sI_F^{x_1+}(w))\nu_i\d s\right.\\
    &&\left.+\int_{F_K^{x_1-}}\left(\frac{\partial u}{\partial x_i}
    -\Theta(\bK,\frac{\partial u}{\partial
    x_i},x_2,x_3)-M_{F_K^{x_1-}}(\frac{\partial u}{\partial x_i})\right)
    (w-\sI_F^{x_1-}(w))\nu_i\d s\right)\\
    &=& \sum _{\bK\in\Tau_h}\sum_{i=1}^3\left(\int_{F_K^{x_1+}}\left(\frac{\partial u}{\partial x_i}
    -R_{\bK}\frac{\partial u}{\partial x_i}\right)
    (w-\sI_F^{x_1+}(w))\nu_i\d s\right.\\
    &&\left.+\int_{F_K^{x_1-}}\left(\frac{\partial u}{\partial x_i}-R_{\bK}\frac{\partial u}{\partial x_i}\right)
    (w-\sI_F^{x_1-}(w))\nu_i\d s\right).
\end{eqnarray*}
Since $R_{\bK}$ and $\sI_F^{x}$ preserves $P_1(\bK)$ and
$P_1(F_K^{x_1})$,
respectively, it
follows from trace theorem and Cauchy-Schwartz inequality, we get
\begin{eqnarray*}
    |E_1|\leq Ch^2||w||_h||u||_{H^3(\O)}.
\end{eqnarray*}
Similarly, we also have
\begin{eqnarray*}
    |E_2|\leq Ch^2||w||_h||u||_{H^3(\O)},\ |E_3|\leq Ch^2||w||_h||u||_{H^3(\O)}.
\end{eqnarray*}
Hence
\begin{eqnarray*}
    \sup_{w\in\NC_0}\frac{|a_h(u,w)-\langle f,w\rangle|}{\|w\|_h}=\sup_{w\in\NC_0}
    \frac{|E_1+E_2+E_3|}{||w||_h}\leq Ch^2 ||u||_{H^3(\O)}.
\end{eqnarray*}
By collecting the above results, we get the following energy-norm error estimate.

\begin{thm}\label{thm:h1}
    Let $u\in H^3(\O)\cap H_0^1(\O)$ and $u_h\in \NC_0$ satisfy
    \eqref{eq:weak} and \eqref{eq:solution}, respectively. Then we have the energy norm
    error estimate:
    \begin{eqnarray*}
        ||u-u_h||_h\leq Ch^2||u||_3.
    \end{eqnarray*}
\end{thm}

By \cor{a} standard Aubin-Nitsche duality argument, \cor{an} $L_2(\O)$-error
estimate can be easily obtained.
\begin{thm}\label{thm:l2}
    Let $u\in H^{3}(\O)\cap H_0^1(\O)$ and $u_h\in\NC_0$ be the solution of
    \eqref{eq:weak} and \eqref{eq:solution}, respectively. Then we have
    \begin{eqnarray*}
        \|u-u_h\|_{0}\leq Ch^{3} \|u\|_3.
    \end{eqnarray*}
\end{thm}

\begin{proof}
Let $\eta_h = \Pi_h u - u_h\in \NC_0$ and consider the dual problem:
\begin{subeqnarray}
-\Delta \psi &=& \eta_h,\quad\O,\\
\psi &=& 0,\quad\p\O.
\end{subeqnarray}
Since $\eta_h \in L^2(\O)$, the elliptic regularity implies that
$\|\psi\|_2 \le C \|\eta_h\|.$ An application of \eqref{eq:hilbert} to the
triangle inequality makes us to prove only
$\|\eta_h \|_h \le  C h^2 \|u\|_3.$

First, we have from Theorem \ref{thm:h1} and \eqref{eq:hilbert} that
\begin{eqnarray}\label{eq:eta-tri}
\|\eta_h \|_h \le \|u-u_h\|_h + \| u - \Pi_h u\|_h \le C h^2 \|u\|_3.
\end{eqnarray}

Following the arguments in the derivation of
the energy estimate, we have
\begin{eqnarray*}
\|\eta_h\|^2 &=& -\sum_{\bK\in\Tau_h} \left(\eta_h,\Delta
\psi\right)_\bK \\
&=& \sum_{\bK\in\Tau_h} \left( \grad \eta_h,  \grad \psi\right)_\bK -
\sum_{\bK\in\Tau_h}  \<\eta_h, \bnu\cdot \grad \psi\>_{\p\bK}\\
&=& a_h(\eta_h,\psi) -
\sum_{\bK\in\Tau_h}  \<\eta_h - \sI_F \eta_h  , \bnu\cdot\grad \psi
 \>_{\p\bK},\\
&=& a_h(\eta_h,\psi) -
\sum_{\bK\in\Tau_h}  \<\eta_h - \sI_F \eta_h  , \bnu\cdot\left(\grad \psi
-M_{F}(\grad \psi)\right)\>_{\p\bK},
\end{eqnarray*}
where $M_F(\nabla \psi)$ denotes the value of
$\nabla \psi$ at the centroid of $F$.
Hence, invoking the elliptic regularity and \eqref{eq:eta-tri}
\begin{eqnarray}
\|\eta_h\|^2 &\le&
|a_h(\eta_h,\psi)| +
\left[\sum_{\bK\in\Tau_h} \left|\eta_h - \sI_F
  \eta_h\right|^2_{0,\p\bK}\right]^{\frac12}
  \left[\sum_{\bK\in\Tau_h} \left|\bnu\cdot\left(\grad \psi -
  M_F(\grad \psi)\right)\right|^2_{0,\p\bK}\right]^{\frac12} \nn \\
&\le& |a_h(\eta_h,\psi)| +
C h^{\frac12}  \|\eta_h\|_h  h^{\frac12} \| \psi \|_2 \nn\\
&\le&  |a_h(\eta_h,\psi)| + C h^3 \|u\|_h~ \|\eta_h\|.\label{eq:etah}
\end{eqnarray}
Thus it remains to bound $ |a_h(\eta_h,\psi)|.$ For this, write
\begin{eqnarray}\label{eq:eta-psi}
a_h(\eta_h,\psi) = a_h(\eta_h,\psi - \Pi_h\psi) +   a_h(\Pi_h u - u, \Pi_h\psi) +  a_h(u - u_h, \Pi_h\psi).
\end{eqnarray}
The first term in \eqref{eq:eta-psi} is bounded as follows:
\begin{eqnarray}
\nn
|a_h(\eta_h,\psi - \Pi_h\psi) | &\le& C \|\eta_h\|_h \|\psi-\Pi_h\psi\|_h \\ &\le&
C h^2 \|u\|_3 h\|\psi\|_2 \le C h^3 \|u\|_3 h\|\eta_h\|.
\label{eq:eta-psi-1}
\end{eqnarray}
Since the second term in \eqref{eq:eta-psi} can be decomposed by
\begin{eqnarray*}
a_h(\Pi_h u - u, \Pi_h\psi) &=&
\sum_{\bK\in\Tau_h}\left(\Pi_h u - u,-\Delta(\Pi_h \psi)\right)_\bK \\
&&\qquad+ \sum_{\bK\in\Tau_h} \<\Pi_h u -
u,\bnu\cdot\grad\Pi_h \psi)\>_{\p\bK}\\
&=& \sum_{\bK\in\Tau_h}\left(\Pi_h u - u,-\Delta\Pi_h \psi\right)_\bK \\
&&\qquad+ \sum_{\bK\in\Tau_h} \<\Pi_h u -
u,\bnu\cdot\left(\grad\Pi_h \psi -M_F\left(\grad\Pi_h
\psi\right)\right)\>_{\p\bK},
\end{eqnarray*}
it can be estimated as follows:
\begin{eqnarray}\label{eq:eta-psi-2}
|a_h(\Pi_h u - u, \Pi_h\psi)| \le C h^3\| u\|_3 \|\psi\|_2
+ C h^{\frac32} \|u\|_3 h^{\frac12} \|\psi\|_2 \le  C h^3\| u\|_3 \|\eta_h\|.
\end{eqnarray}
The third term in \eqref{eq:eta-psi} is bounded in the same fashion as
\begin{eqnarray}
|a_h(u - \Pi_h u, \Pi_h\psi ) | \le  C h^3\| u\|_3 \|\eta_h\|.
\label{eq:eta-psi-3}
\end{eqnarray}
Collecting \eqref{eq:eta-psi-1}--\eqref{eq:eta-psi-3} and plugging into
\eqref{eq:eta-psi} combined with \eqref{eq:etah}, one sees that
the theorem follows by deviding both sides by $\|\eta_h\|.$
\end{proof}

\section{Error estimates of other Smith-Kidger elements}
In this section we claim that the approximation of the solutions of the second-order elliptic problem
with Type 2 and 5 Smith-Kidger elements is also convergent in
optimal order. In these case, it is easy to check that the
orthogonality in Lemma \ref{lem:orth} holds. The difference during
the proof lies in the construction of the interpolation operator.
For the second type element, the interpolation of
$\sI_F^{x_1}$,
$\sI_F^{x_2}$, $\sI_F^{x_3}$ should be $\Span\{1,\y,\z,\y\z,\y^2\}$,
$\Span\{1,\x,\z,\x\z,\z^2\}$ and $\Span\{1,\x,\y,\x\y,\x^2\}$, respectively.
And for the fifth type element, the corresponding interpolation spaces
should be taken as $\Span\{1,\y,\z,\y\z,\z^2+\y^2\}$,
$\Span\{1,\x,\z,\x\z,\x^2+\z^2\}$ and $\Span\{1,\x,\y,\x\y,\x^2+\y^2\}$,
respectively.

For Type 6 element, the orthogonality in Lemma \ref{lem:orth} does
not hold, but the Eq. \eqref{eq:Fx} hods. Thus, we have
\begin{eqnarray*}
    |E_1| &=& \left|\sum_{\bK\in\Tau_h}\left(\int_{F_K^{x_1+}}\frac{\p u}{\p\bnu}
    (w-\sI_F^{x_1+}(w))\,\d s +\int_{F_K^{x_1-}}\frac{\p u}{\p\bnu}
    (w-\sI_F^{x_1-}(w))\,\d s\right)\right|\\
    &=& \left|\sum_{\bK\in\Tau_h}\left(\int_{F_K^{x+}}\left(\frac{\p u}{\p\bnu}-P_{\bK}^0(\frac{\p u}{\p\bnu})\right)
    (w-\sI_F^{x_1+}(w))\,\d s\right.\right.\\
    &&\left.\left.+\int_{F_K^{x_1-}}\left(\frac{\p u}{\p\bnu}-P_{\bK}^0(\frac{\p u}{\p\bnu})\right)
    (w-\sI_F^{x_1-}(w))\,\d s\right)\right|\\
    &\leq & Ch||u||_2||w||_h,
\end{eqnarray*}
where $$P_{\bK}^0\left(\frac{\p u}{\p\bnu}\right)=\frac{1}{|\bK|}\int_{\bK}\frac{\p u}{\p\bnu} \,\d s,$$
and $|\bK|=\int_{\bK}\,\d s$.
By a similar derivation, we will get
\begin{thm}
    Let $u\in H^2(\O)\cap H_0^1(\O)$ satisfy
    \eqref{eq:weak} and $u_h$ be the solution of \eqref{eq:solution} with
    the sixth type element. Then we have the energy norm
    error estimate:
    \begin{eqnarray*}
        ||u-u_h||_h\leq Ch||u||_2,\\
        ||u-u_h||_0\leq Ch^{2} \|u\|_2.
    \end{eqnarray*}
\end{thm}

\section{A new 14-node brick element}
In this section, we present a new element with 14-node. The degrees
of freedom are the same with those in Smith-Kidger element and
Meng-Sheen-Luo-Kim element. But the shape function space is taken as
$P_2\oplus \Span\{\cor{\x\y\z,} \x(\y^2+\z^2),\y(\x^2+\z^2),\z(\x^2+\y^2)\}$. 
\cor{
Denote the corresponding higher-degree polynomials to those in \eqref{eq:r} as
follows: 
\begin{eqnarray}\label{eq:new-r}
\begin{split}
r_0(\hx_1,\hx_2,\hx_3) = \hx_1\hx_2\hx_3,~ r_1(\hx_1,\hx_2,\hx_3) =
\frac12\hx_1(\hx_2^2+\hx_3^2)
,\\
r_2(\hx_1,\hx_2,\hx_3) =\frac12\hx_2(\hx_1^2+\hx_3^2),  ~r_3(\hx_1,\hx_2,\hx_3) = \frac12\hx_3(\hx_1^2+\hx_2^2).
\end{split}
\end{eqnarray}
Then, again equipped with the 8 vertex values plus 6 face integrals DOFs,
the basis functions corresponding to the vertices and face-integrals are given
exactly same as the formulae \eqref{eq:phiV}, \eqref{eq:phiF}, and \eqref{eq:ell-q-r}.
} 

\cor{In order to analyze convergence}, we need to verify the orthogonality in
Lemma \ref{lem:orth} and Eq. \eqref{eq:Fhx}.
The orthogonality can be checked directly \cor{as in} the proof in Lemma
\ref{lem:orth}. \cor{In order to check} Eq. \eqref{eq:Fhx}, \cor{it is enough
  to
define}
the corresponding interpolation spaces as $\Span\{1,\y,\z,\y\z,\z^2+\y^2\}$, $\Span\{1,\x,\z,\x\z,\x^2+\z^2\}$
and $\Span\{1,\x,\y,\x\y,\x^2+\y^2\}$, respectively. Thus, \cor{by following the
same argument as in the previous sections}, we also get optimal convergence 
for the
second-order elliptic problems. That is, in this case, Theorems
\ref{thm:h1} and \ref{thm:l2} hold. 

\section{Further remarks and conclusions}
In this paper, we have proved that for second-order elliptic
problems, the Smith-Kidger element of type 1, 2 and 5 can obtain
optimal convergence order both in energy norm and $L_2(\O)$ norm,
while the sixth type element loses one order of accuracy in each
norm. In the proof, the key points lie in that they have weak
orthogonality (Lemma \ref{lem:orth}) and satisfy Eq. \eqref{eq:Fhx}.
In \cite{2013-Meng-p-}. \cor{We} also proposed another kind of DOFs, that
is, the \cor{values} at the eight vertices and the integration {values over} six
faces. Indeed, it is easy to check that Type 1, 2, 5 and the new
element presented in this paper \cor{give} optimal convergence \cor{orders}
for second-order elliptic problems. Besides, we can show that 
\cor{if the face-centroid values DOFs are replaced by the face integrals DOFs,}
Type 6 \cor{element} 
also \cor{are of optimal-order} convergence \cor{owing to} a weak
orthogonality, \cor{thus} improving one order accuracy.

\bibliographystyle{abbrv}

    \end{document}